\documentclass[a4paper,10pt]{amsart}
\usepackage{graphicx}
\usepackage{caption}
\usepackage{subcaption}
\usepackage{float}
\usepackage{fancybox}
\usepackage{verbatim}
\usepackage{listings}
\usepackage{fancyvrb}
\usepackage{amsfonts}
\usepackage{color}
\usepackage{colortbl}
\usepackage{amsmath}
\usepackage{dsfont}
\usepackage{epsfig}
\usepackage{amssymb}
\usepackage{amsbsy}
\usepackage{amsmath}
\usepackage{enumerate}
\usepackage{stmaryrd}
\usepackage{mathtools}
\usepackage{tikz}
\usepackage{bbm}
\usepackage{fancyhdr}
\usepackage{amsthm}
\usepackage[margin=1in]{geometry}
\usepackage{hyperref}
\usepackage{bbm}
\usepackage{soul}
\hypersetup{
	linktocpage = true,
	colorlinks = true,
	linkcolor = {blue},
	urlcolor = {blue},
	citecolor = {blue}}
\newtheorem{prop}{Proposition}
\newtheorem{theorem}{Theorem}

\newtheorem{lemma}{Lemma}

\newtheorem{rem}{Remark}
\newtheorem{definition}{Definition}

\def\N {\mathbb{N}}

\def\R {\mathbb{R}}

\def\d {\mathrm{d}}

\newcommand{\uin}{u^{\mathrm{in}}}
\newcommand{\vin}{v^{\mathrm{in}}}

\begin{document}
\title[A nonlinear diffusion equation arising in game theory]{Self-similar
solutions, regularity and time asymptotics for a nonlinear diffusion
equation arising in game theory}
\author{Marco Antonio Fontelos}
\author{Nastassia Pouradier Duteil}
\author{Francesco Salvarani}
\address{M.A.F.: Instituto de Ciencias Matem\'aticas,(ICMAT,
CSIC-UAM-UCM-UC3M),Campus de Cantoblanco, 28006 Madrid, Spain}
\email{marco.fontelos@icmat.es}
\address{N.P.D.: Sorbonne Universit\'e, Inria, Universit\'e Paris-Diderot
SPC, CNRS, Laboratoire Jacques-Louis Lions, Paris, France}
\email{nastassia.pouradier\_duteil@sorbonne-universite.fr}
\address{F.S.: L\'eonard de Vinci P\^ole Universitaire, Research Center,
92916 Paris La D\'efense, France \& Dipartimento di Matematica ``F.
Casorati'', Universit\`a degli Studi di Pavia, Via Ferrata 1, 27100 Pavia,
Italy}
\email{francesco.salvarani@unipv.it}

\begin{abstract}
In this article, we study the long-time asymptotic properties of a non-linear and non-local equation of diffusive type which describes the 
rock-paper-scissors game in an interconnected population. We fully characterize the self-similar solution and then prove that the solution of the initial-boundary value problem converges to the self-similar profile with an algebraic rate.
\end{abstract}

\maketitle

\section*{Introduction}


The rock-paper-scissors game is not only one of the classical examples in
game theory, but it arises also in other contexts, such as bacterial ecology
and evolution, where it has been extended to the scale of an entire
population. In several situations, indeed, the rock-paper-scissors game
allows to model cyclic competition between species and the stabilization of
bacteria populations \cite{2002Natur, 2004Natur, Patk2012}, i.e. when three
species coexist and there is cyclic domination of the first species on the
second one, of the second species on the third one, and of the third species
on the first one. Moreover, some applications of this game have been
proposed in evolutionary game theory, for example to explain the coexistence
or extinction of species \cite{Shi2010BasinsOA} or male reproductive
strategies \cite{SL96}.

This justifies the importance of having a description of the 
rock-paper-scissors dynamics at the mesoscopic (i.e. kinetic) and macroscopic levels, where the
population is described by a density function: it allows the description of the global dynamics
without needing to take into account the individual situations, and is therefore well adapted
for population with a high number of individuals.

A kinetic version of the rock-paper-scissors game has been studied in \cite{MR4128577}. 
This situation involves a population of players who form
temporary pairs through random encounters. The two members of a
pair play the game once, then look for another contestant to play with, and
so on. The independent variables are the time $t\in\R_+$ and an individual
exchange variable $x$ (which may correspond to the wealth of individuals, if
the game involves agents exchanging a certain amount of money). In the case
of a fully interconnected population, assuming that there are no forbidden
pairs and that players continue to play as long as their wealth allows, the
corresponding kinetic model introduced in \cite{MR4128577} has the form of
an integro-differential equation on the half-line $\R_+= [0,+\infty)$, with a boundary
condition in $x=0$. 

By assuming that players increase the frequency of the
game by a factor of $\varepsilon^{-1}$, (with $\varepsilon >0$) and, at the
same time, reduce the amount played in each iteration of the game by a
factor of $\varepsilon$, in the limit $\varepsilon \to 0$ the authors of
\cite{MR4128577} obtain a non-linear and non-local partial differential
equation at the classical macroscopic level.

In particular, the limiting initial-boundary value problem for the unknown 
$u \ : \ \R_+\times \R_+\to \R$, which represents the density of agents with
wealth $x\in\R_+$ at time $t \in\R_+$, is the following:
\begin{eqnarray}  \label{eq:cont1}
\displaystyle \partial_t u(t,x) = \left(\int_{\R_+} u(t,z) \mathrm{d}z\right)
\partial^2_{x}u(t,x)\quad & \text{ for a.e. } (t,x)\in \R^*_+\times \R^*_+ \\%
[4pt]
\label{eq:cont2}
u(t,0) = 0 \quad & \text{ for a.e. } t\in \R^*_+\\[10pt]
\label{eq:cont3}
u(0,x) = u^{\mathrm{in}}(x) \quad &\text{ for a.e. } x\in \R_+,
\end{eqnarray}
where $u^{\mathrm{in}}\in L^1(\R_+)\cap L^\infty(\R_+)$ and $\R^*_+=(0,+\infty)$.

Existence and uniqueness of a very weak solution of \eqref{eq:cont1}-
\eqref{eq:cont3} have been proven by means of a compactness argument in \cite{MR4128577}. However, several open questions on this problem are still waiting for an answer. In this
article we study two open questions about the initial-boundary value problem 
\eqref{eq:cont1}-\eqref{eq:cont3}, namely the regularity of the problem and
the intermediate asymptotics with respect to a suitable self-similar
solution, which we will precisely identify. We stress that the asymptotic
behavior is one of the main questions on diffusion equations -- see the
review article \cite{vazquez2017asymptotic} and the references therein.

Equation \eqref{eq:cont1} has a mathematical structure that is essentially
non-local. It can be interpreted as a heat equation, whose diffusivity
coefficient depends on the integral of the solution itself (i.e. the total
mass, in the case of non-negative solutions), which is a typical global
quantity of the system. Because of the peculiar structure of the
nonlinearity in \eqref{eq:cont1}, our methods of proof are sometimes 
close to those used in the study
of linear equations \cite{MR1183805} but, in several points, the need
of approaches designed for studying non-linear
equations are necessary (see, for example, \cite{MR959221}).

More specifically, in this article we prove that, similarly to the heat equation, there exists an instantaneous
gain in regularity. We moreover characterize the self-similar solutions of the problem and identify the convergence speed to
the intermediate asymptotic profile under some conditions on the initial condition which we precisely characterize.
We note that the algebraic convergence speed is a consequence of the non-local structure of the problem.

The structure of this article is the following. The study of the regularity of the problem, together with other basic properties of its
solution, are detailed in Section \ref{s:br}.
Then, in Section \ref{s:self}, we study the long-time convergence of the solution toward the self-similar solution.
We illustrate our study numerically in Section \ref{s:num} and, in the Appendix, we treat the convergence to the self-similar solution in the case
of a bounded interval.


\section{Basic results}
\label{s:br}

In this section we deduce and collect some basic results about the initial-boundary
value problem \eqref{eq:cont1}-\eqref{eq:cont3}.

\subsection{Weak formulation}
We first define the very weak formulation of \eqref{eq:cont1}-%
\eqref{eq:cont3} as follows:

\begin{definition}
\label{def:veryweak}
Let $T>0$. A measurable function $u\in L^1([0,T]\times\R_+)$ is
said to be a very weak solution of the initial-boundary value problem %
\eqref{eq:cont1}-\eqref{eq:cont3} if it satisfies
\begin{equation}  \label{eq:veryweaksol}
\begin{array}{l}
\displaystyle \int_0^T \!\! \int_{\R_+}\!\! u(t,x) \partial_t\varphi(t,x) \,%
\mathrm{d} x \,\mathrm{d} t + \int_0^T \!\! \left(\int_{\R_+} \!\! u(t,x_*)
\,\mathrm{d} x_*\right) \int_{\R_+} \!\! u(t,x) \partial^2_{x} \varphi(t,x)
\,\mathrm{d} x \,\mathrm{d} t \\
\displaystyle + \int_{\R_+} \!\!u^{\mathrm{in}}(x)\varphi(0,x) \,\mathrm{d}
x=0%
\end{array}%
\end{equation}
for all $\varphi\in C^1([0,T];C_c^2(\R))\cap L^\infty([0,T]\times\R)$,
such that $\varphi(T,x)=0$ for all $x\in\R_+$, $\varphi_x(t,0)=0$ 
for all $
t\in [0,T]$, where $ C_c^2(\R)$ is the space of $C^2$ compactly supported functions on $\R$.
\end{definition}

Existence and uniqueness of a very weak solution to \eqref{eq:cont1}-\eqref{eq:cont3} was proven in \cite{MR4128577}. Moreover, one can prove that the solution is bounded by the $L^\infty$ norm of the initial data, and if the initial data is non-negative, the solution remains non-negative for all time. The precise results are recalled in the following theorem (see \cite{MR4128577}):

\begin{theorem}
\label{teo:exun} Consider the initial-boundary value problem \eqref{eq:cont1}%
-\eqref{eq:cont3}, with initial condition $u^{\mathrm{in}}\in L^1(\R_+)\cap
L^\infty (\R_+)$ and such that $u^{\mathrm{in}}\geq 0$ for a.e. $x\in\R_+$.
Let $T > 0$. Then, it has a unique very weak solution, which belongs to $%
L^1((0,T)\times \R_+)\cap L^\infty ((0,T)\times \R_+)$. Moreover, $\Vert
u(t, \cdot ) \Vert_{L^\infty(\R_+)} \leq \Vert u^{\mathrm{in}}
\Vert_{L^\infty(\R_+)}$ for a.e. $t\in (0,T)$. Lastly, the solution is
non-negative, i.e. $u(t,x)\geq 0$ for a.e. $t\in(0,T)$ and for a.e. $x\in\R%
_+ $.
\end{theorem}


\subsection{Improved regularity and positivity}

\label{ss:imp} 

Let $u$ be the very weak solution of the initial-boundary value problem %
\eqref{eq:cont1}-\eqref{eq:cont3}. Then, it is possible to consider its
antisymmetric extension $v$, defined for all $x\in\R$, such that
\begin{equation}
\label{e:antisym}
v(t,x)= u(t,x)\mathds{1}_{x\geq 0} - u(t,-x)\mathds{1}_{x\leq 0},
\end{equation}
for a.e. $t\in (0,T)$.

Consequently, $v$ solves (in the very weak sense) the following auxiliary
initial value problem for the unknown $v\ :\ \R_+\times \R \to \R$
\begin{equation}
\begin{cases}
\label{eq:auxv} \displaystyle \partial_t v(t,x) = \left(\int_{\R^+} v(t,z) \,%
\mathrm{d} z\right) \partial^2_{x}v(t,x)\quad & \text{ for a.e. } (t,x)\in \R%
_+\times \R \\[10pt]
v(0,x) = v^{\mathrm{in}}(x) \quad & \text{ for a.e. } x\in \R,%
\end{cases}%
\end{equation}
where $v^{\mathrm{in}}(x)=u^{\mathrm{in}}(x)\mathds{1}_{x\geq 0} - u^{%
\mathrm{in}}(-x)\mathds{1}_{x\leq 0} $ for a.e. $x\in\R$. Because of the
regularity conditions on the initial data,
we immediately deduce that
$v^{\mathrm{in}}\in L^1(\R)\cap L^\infty (\R)$.
This antisymmetric extension of $u$ will allow us to prove the following result:

\begin{prop}
\label{p:reg} Let $u$ be the very weak solution of the initial-boundary
value problem \eqref{eq:cont1}-\eqref{eq:cont3}, with initial and boundary
conditions satisfying the hypotheses of Theorem \ref{teo:exun}. Then, $u\in
C^\infty((0,T)\times\R^*_+)$. Moreover, $u$ admits the following
semi-explicit representation:
\begin{equation}
\begin{array}{ll}
\displaystyle u(t,x)= \left (4\pi \int_{0}^t\int_{\R_+} u(\theta,z) \,%
\mathrm{d} z \,\mathrm{d} \theta\right)^{-1/2} \times &  \\[10pt]
\displaystyle \int_{\R_+} \!\! u^{\mathrm{in}} (y) \left \{ \exp \!\left
[-(x\!-\!y)^2\left(4\int_{0}^t\!\int_{\R_+} \!\!u(\theta,z) \,\mathrm{d} z \,%
\mathrm{d} \theta\right)^{\!\!-1}\right] \!-\! \exp\! \left
[-(x\!+\!y)^2\left(4\int_{0}^t\!\int_{\R_+} \!\!u(\theta,z) \,\mathrm{d} z \,%
\mathrm{d} \theta\right)^{\!\!-1}\right] \right \} \,\mathrm{d} y. &
\end{array}
\label{eq:usol}
\end{equation}
\end{prop}
\begin{proof}
We consider the auxiliary problem \eqref{eq:auxv}. We have not yet proven
that $v$ is the unique solution of \eqref{eq:auxv}, but we know, by
construction, that it exists and belongs to $L^1((0,T)\times \R  )\cap
L^\infty ((0,T)\times  \R  )$, because of the results proved in \cite{MR4128577}.

We hence introduce the spatial Fourier transform 
$$\hat v : L^1((0,T)\times \R)\cap
L^\infty ((0,T)\times \R) \to  L^2((0,T)\times \R),
$$
\par\noindent
which is meaningful because of the regularity hypotheses on $v$. We use the
following convention for the Fourier transform of a function and for its
inverse:
\begin{equation*}
\forall \ \xi \in \mathbb{R}, \quad\displaystyle{\hat {v}}(t,\xi )=\int _{\R%
}v(t,x)\ e^{-2\pi i \xi x}\d x,
\end{equation*}
and
\begin{equation*}
\forall \ x \in \mathbb{R}, \quad \displaystyle{\ {v}}(t,x )=\int _{\R}{%
\hat {v}}(t,\xi ) \ e^{2\pi i \xi x}\d \xi.
\end{equation*}
By applying the Fourier transform with respect to the $x$ variable to all
terms in the Cauchy problem \eqref{eq:auxv}, we deduce a problem for the
Fourier transform $\hat v$ of the solution, i.e. we
obtain
\begin{equation}  \label{eq:auxqq}
\begin{cases}
\displaystyle \partial_t \hat v(t,\xi) =-4\pi^2 \xi^2\left(\int_{\R_+}
v(t,z) \,\mathrm{d} z\right) \hat v(t,\xi)\quad & \text{ for a.e. }
(t,\xi)\in \R_+\times \R \\[10pt]
v(0,\xi) = \hat v^{\mathrm{in}} (\xi)= \mathcal{F}(v^{\mathrm{in}})(\xi)\quad & \text{ for a.e. } x\in \R.%
\end{cases}%
\end{equation}
This auxiliary problem can be integrated in time, allowing to deduce the
integral form of the initial-boundary value problem \eqref{eq:auxv}:
\begin{equation}  \label{eq:aux1}
\hat v (t,\xi)= \hat v^{\mathrm{in}} (\xi) \exp\left [-4\pi^2 \xi^2\int_0^t
\left(\int_{\R_+} v( \theta  ,z) \,\mathrm{d} z\right) \d\theta \right]
\end{equation}
for all $\xi\in\R$.

Thanks to the regularity of $u^{\mathrm{in}}$, we have that
$\hat v^{\mathrm{in}}\in L^\infty(\R)$.
Hence, by Formula \eqref{eq:auxqq}
the decay to zero of $\hat v$
when $\xi$ tends to $+\infty$ is faster than polynomial, for any degree of
the polynomial. Consequently, $v\in L^1((0,T); C^\infty(\R))$ (see, for example, {\cite{MR1157815}}).

By applying the inverse Fourier transform to 
 the second factor of Equation \eqref{eq:aux1}, we
find
\begin{equation*}
\begin{split}
\mathcal{F}^{-1}& \left(\exp\left (-4\pi^2 \xi^2\int_0^t
\left(\int_{\R_+} v( \theta  ,z) \,\mathrm{d} z\right) \d\theta \right)\right) \\
& = 
\left(4\pi \! \int_{0}^t \!\! \int_{\R_+} \! 
v(\theta,z) \,\mathrm{d} z \,\mathrm{d} \theta \right)^{-1/2} 
\exp \left [-x^2\left(4\int_{0}^t\!\!
\int_{\R_+} \! v(\theta,z)\,\mathrm{d} z \,\mathrm{d}
\theta\right)^{-1}\right],
\end{split}
\end{equation*}
so that
\begin{equation*}
v (t,x) =  \left (4\pi \! \int_{0}^t \!\! \int_{\R_+} \! 
 v(\theta,z) \,\mathrm{d} z \,\mathrm{d} \theta \right)^{-1/2}  
\int_{\R}\! v^{\mathrm{in}} (y) \exp \left [-(x-y)^2\left(4\int_{0}^t\!\!
\int_{ \R_+} \!  v(\theta,z)\,\mathrm{d} z \,\mathrm{d}
\theta\right)^{-1}\right] \,\mathrm{d} y.
\end{equation*}
for all $x\in\R$.

We note that, if $v\in L^1((0,T); C^\infty(\R))$, then the right-hand side
of the previous equation is, in fact, a quantity belonging to $C((0,T); C^\infty(\R))$. By a bootstrap argument \cite{MR0241822, MR1157815}, we immediately deduce that $v\in
C^\infty((0,T)\times\R)$.

In particular, when $x> 0$, we can write the previous expression in the
following way:
\begin{equation*}
v (t,x)=\left (4\pi \!\int_{0}^t\int_{\R_+}  v(\theta,z)\,\mathrm{%
d} z \,\mathrm{d} \theta\right)^{-1/2} \times
\end{equation*}
\begin{equation*}
\int_{\R_+}\!\! \! u^{\mathrm{in}}(y) \!\left \{\! \exp \!\left
[-(x\!-\!y)^2\!\left(4\!\int_{0}^t\!\int_{\R_+} \!\! v(\theta,z) \,%
\mathrm{d} z \,\mathrm{d} \theta\right)^{\!\!-1}\right] \!-\! \exp\! \left
[-(x\!+\!y)^2\!\left(4\!\int_{0}^t\!\int_{\R_+} \!\! v(\theta,z) \,%
\mathrm{d} z \,\mathrm{d} \theta\right)^{\!\!-1}\right]\! \right \}\! \,%
\mathrm{d} y.
\end{equation*}
Moreover, for $x=0$, we have that $v (t,0)=0$ for all $t\in(0,T)$.
Furthermore, $v$ is clearly strictly positive for all $x>0$ provided that $%
u^{\mathrm{in}}$ is non-negative for a.e. $x\in\R_+$. By comparing %
\eqref{eq:auxv} and \eqref{eq:cont1}-\eqref{eq:cont3}, we deduce that 
$\tilde{u}= v \mathds{1}_{x\geq 0}$ satisfies 
the initial-boundary value problem \eqref{eq:cont1}-\eqref{eq:cont2}  with initial value $\tilde{u}(0,\cdot)=\vin \mathds{1}_{x\geq 0}$.

Because of the uniqueness of the very weak solution of \eqref{eq:cont1}-\eqref{eq:cont3} (see 
Theorem \ref{teo:exun}), we deduce that
$\tilde{u}(t,x)= v(t,x)\mathds{1}_{x\geq 0}=u(t,x)$
\color{black}
for a.e. $(t,x)\in(0,T)\times\R_+$. Consequently,
 the previous
computation allows to obtain a semi-explicit representation of $u$:
\begin{equation*}
\begin{array}{ll}
\displaystyle u(t,x)= \left (4\pi \int_{0}^t\int_{\R_+} u(\theta,z) \,%
\mathrm{d} z \,\mathrm{d} \theta\right)^{-1/2} \times &  \\[10pt]
\displaystyle \int_{\R_+} \!\! u^{\mathrm{in}} (y) \left \{ \exp \!\left
[-(x\!-\!y)^2\left(4\int_{0}^t\!\int_{\R_+} \!\!u(\theta,z) \,\mathrm{d} z \,%
\mathrm{d} \theta\right)^{\!\!-1}\right] \!-\! \exp\! \left
[-(x\!+\!y)^2\left(4\int_{0}^t\!\int_{\R_+} \!\!u(\theta,z) \,\mathrm{d} z \,%
\mathrm{d} \theta\right)^{\!\!-1}\right] \right \} \,\mathrm{d} y. &
\end{array}%
\end{equation*}
Finally, we underline that $u\in C^\infty((0,T)\times\R^*_+)$ because $u$
inherits the regularity properties of $v$.
\end{proof}


\subsection{Some quantitative bounds}


The first step in our analysis consists in proving some uniform estimates.
 In all that follows, we will assume that $\uin\geq 0$.

\begin{prop}
\label{p:mass} Let $u$ be the solution of the initial-boundary value problem %
\eqref{eq:cont1}-\eqref{eq:cont3} and let
\begin{equation*}
M :t\mapsto\int_{\R_+} u(t,x) \,\mathrm{d} x.
\end{equation*}
Then $M$ is a decreasing function of time. In particular, $%
M\in C^\infty((0,T))$ and, for all $t\in \R_+$,
\begin{equation*}
M (t) \leq M (0)= \int_{\R_+} u^{\mathrm{in}}(x) \,\mathrm{d} x.
\end{equation*}
\end{prop}

\begin{proof}
The result is a direct consequence of the regularity proven in Proposition %
\ref{p:reg}. Integrating in $x$ Equation \eqref{eq:cont1}, it holds
\begin{equation*}
M^{\prime }(t) = \left(\int_0^{+\infty}\partial_x^2 u(t,z) \mathrm{d}z \right)
M(t) = \left(\lim_{x\rightarrow+\infty}\partial_x u(t,x) - \partial_x u(t,0)\right) M(t).
\end{equation*}
By differentiating both sides of Equation \eqref{eq:usol} with respect to $x$%
, we obtain
\begin{equation*}
\begin{array}{ll}
\displaystyle \partial_x u(t,x)\! = \! -\frac 2{\sqrt{\pi}} \left (4\!
\int_{0}^t\int_{\R_+} \! u(\theta,z) \,\mathrm{d} z \,\mathrm{d}
\theta\right)^{-3/2} \!\!\!\! \int_{\R_+} \!\! \!u^{\mathrm{in}} (y)(x-y)
\exp\! \left [-(x\!-\!y)^2\! \left(4\! \int_{0}^t\!\int_{\R_+}
\!\!u(\theta,z) \,\mathrm{d} z \,\mathrm{d} \theta\right)^{\!\!-1}\right] \!%
\mathrm{d} y &  \\[10pt]
+\displaystyle \frac 2{\sqrt{\pi}} \left (4 \int_{0}^t\int_{\R_+}
u(\theta,z) \,\mathrm{d} z \,\mathrm{d} \theta\right)^{-3/2} \int_{\R_+}
\!\! u^{\mathrm{in}} (y)(x+y) \exp\! \left
[-(x\!+\!y)^2\left(4\int_{0}^t\!\int_{\R_+} \!\!u(\theta,z) \,\mathrm{d} z \,%
\mathrm{d} \theta\right)^{\!\!-1}\right] \,\mathrm{d} y. &
\end{array}%
\end{equation*}
We deduce that $\partial_x u(t,0)\geq 0$ and $\lim_{x\to+\infty} \partial_x
u(t,x) =0$ for all $t\in(0,T) $.
The thesis follows directly.
\end{proof}

For future purposes, we introduce the spatial first moment $M_1:\R_+\to\R$
such that, for $u$ solution of \eqref{eq:cont1}-\eqref{eq:cont3}, it holds
\begin{equation*}
M_1(t):=\int_{\R_+}x \, u (t,x)\,\mathrm{d} x, \qquad \text{for all } t\in\R%
_+.
\end{equation*}
Moreover, we introduce the spatial second moment $M_2:\R_+\to\R$ such that,
for $u$ solution of \eqref{eq:cont1}-\eqref{eq:cont3}, it holds
\begin{equation*}
M_2(t):=\int_{\R_+}x^2 \, u (t,x)\,\mathrm{d} x, \qquad \text{for all } t\in%
\R_+.
\end{equation*}
From here onward, we consider initial data with bounded spatial first
moment, i.e. we suppose that the following property is satisfied.
\begin{definition}
\label{hyp:firstmoment} 
The initial condition $u^{\mathrm{in}}\in L^1(\R_+)\cap L^\infty(\R_+)$ is
admissible if and only if 
\begin{equation*}
M_1(0)=\int_{\R_+}xu^{\mathrm{in}}(x)\,\mathrm{d} x<+\infty.
\end{equation*}
\end{definition}

The following result holds.

\begin{prop}
Let $u$ be a (strong) solution of \eqref{eq:cont1}-\eqref{eq:cont3}, and
suppose that $u^{\mathrm{in}}$  is admissible (see Definition \ref{hyp:firstmoment}).
Then the spatial first moment of $u
$ is conserved, i.e. for all $t\in\R_+$,
\begin{equation*}
\int_{\R_+} x u(t,x) \,\mathrm{d} x= \int_{\R_+} x u^{\mathrm{in}}(x) \,%
\mathrm{d} x.
\end{equation*}
\end{prop}

\begin{proof}
Consider $u:\R_+\times\R_+\to\R$ a solution to \eqref{eq:cont1}-%
\eqref{eq:cont3}, and let $M$ be the total mass defined in Proposition \ref%
{p:mass}. Let $v:\R_+\times\R\to\R$ be the antisymmetric extension of $u$
defined in Equation \eqref{e:antisym} and studied in Subsection \ref{ss:imp}.
Notice that
\begin{equation*}
\int_{\R} x v(t,x) \,\mathrm{d} x = \int_{\R_+} x u(t,x)\,\mathrm{d} x -
\int_{\R_-} x u(t,-x)\,\mathrm{d} x = \int_{\R_+} x u(t,x)\,\mathrm{d} x +
\int_{\R_+} x u(t,x)\,\mathrm{d} x = 2 \int_{\R_+} x u(t,x)\,\mathrm{d} x.
\end{equation*}
Now, let $a:\R_+\to\R_+$ be defined by
\begin{equation*}
a:t\mapsto \int_0^t M (\tau)\,\mathrm{d} \tau,
\end{equation*}
and define $\tilde{v}:\R_+\times\R\to\R$ as follows: $\tilde{v}%
(a(t),x)=v(t,x)$ for all $(t,x)\in\R_+\times\R$. Then $\tilde{v}$ is a
solution to
\begin{equation*}
\begin{cases}
\displaystyle \partial_{a} \tilde{v}(a,x) = \tilde{v}_{xx}(a,x)\quad
& (a,x)\in \mathcal{T}\times \R, \\[10pt]
\tilde{v}(0,x) = v^{\mathrm{in}}(x) \quad & \text{ for a.e. } x\in \R,%
\end{cases}%
\end{equation*}
where
\begin{equation*}
\mathcal{T} =\left( 0, \int_0^{+\infty} M (\tau)\,\mathrm{d} \tau \right).
\end{equation*}
Hence, $\tilde{v}$ satisfies the Cauchy problem for the heat equation on the
real line, at least in the time interval $\mathcal{T}$. Since $M(0)>0$ and $%
M $ is continuous with respect to $t\in\R_+$, we deduce that
\begin{equation*}
\int_0^{t} M (\tau)\,\mathrm{d} \tau>0 \quad \text{for all }t>0.
\end{equation*}
At this point, we do not know if
\begin{equation*}
\lim_{t\to+\infty} \int_0^{t} M (\tau)\,\mathrm{d} \tau =+\infty \quad \text{%
or} \quad \lim_{t\to+\infty} \int_0^{t} M (\tau)\,\mathrm{d} \tau <+\infty.
\end{equation*}
However, this is not a problem in our case. It is enough to know that the
first moment of $\tilde{v}$, i.e.
\begin{equation*}
\int_\R x\tilde{v}(\cdot,x)\,\mathrm{d} x,
\end{equation*}
is conserved at least in $\mathcal{T}$. Hence, the first moment of $v$ is
also conserved, and
\begin{equation*}
\int_{\R_+} x u(t,x)\,\mathrm{d} x = \frac{1}{2}\int_{\R} x v(t,x)\,\mathrm{d%
} x = \frac{1}{2}\int_{\R} x \tilde{v}(\tau(t),x)\,\mathrm{d} x
\end{equation*}
is also conserved.
\end{proof}

 Because the first moment is conserved, from here onwards, we
will denote by $M_1$ its value, defined by $M_1=M_1(0)=M_1(t)$ for all $t\in%
\R_+$.

\section{Self-similar solutions and large-time asymptotics}
\label{s:self}


In this section, we
consider the initial-boundary value problem \eqref{eq:cont1}-\eqref{eq:cont3}, 
 and always suppose that the initial data are admissible (i.e. we suppose that $u^{\mathrm{in}}$ satisfies Definition
\ref{hyp:firstmoment}).

Note that the final time $T$ appearing in the statement of Theorem \ref{teo:exun} is finite, but arbitrary, so that the large-time asymptotics of the problem makes sense.


Let $\mu\in\R$. We
look for self-similar solutions $%
g_\mu$ of the form
\begin{equation*}
u(t,x)=t^{\mu-1}g_\mu(x/t^{\mu}),
\end{equation*}
so that the mass of the solution $u$ satisfies for all $t\in\R_+$:
\begin{equation*}
\int_{\R_+} u(t,z)\mathrm{d}z = t^{2\mu-1} \int_{\R_+} g_\mu(\xi)\mathrm{d}\xi.
\end{equation*}
Then $g_\mu(\eta )$ satisfies the following non-local differential equation:
for all $\eta\in\R_+$,
\begin{equation*}
(\mu-1)g_\mu(\eta)-\mu \eta g_{\mu }^{\prime }(\eta)=\left( \int_{0}^{\infty
}g_\mu(s)\,\mathrm{d} s\right) g_{\mu}^{\prime \prime }(\eta).
\end{equation*}
We apply the rescaling 
$$
\eta :\xi\mapsto \left( \int_{0}^{+\infty}g_\mu(s)\,
\mathrm{d} s\right) ^{{1}/{2}}\xi,
$$ 
and denote by 
$$
f_\mu:\xi\mapsto
g_\mu\left(\left(\int_0^{+\infty} g_\mu(s)\,\mathrm{d}s\right)^{1/2}\xi\right)
$$
the solution to the simplified differential equation:
\begin{equation}
(\mu-1)f_\mu(\xi )-\mu\xi f_{\mu }^{\prime }(\xi )=f_{\mu}^{\prime \prime
}(\xi ).
\label{ss1}
\end{equation}
Its relation with $u$ is given by: 
\begin{equation}
u(t,x)=t^{\mu-1}f_\mu\left(\left(\int_{\R_+} u(t,z)\mathrm{d}z\right)^{-\frac{1}{2}}
\frac{x}{t^{\frac{1}{2}}}\right) = t^{\mu-1}f_\mu\left(\left(\int_{\R_+}
f_\mu(s)ds\right)^{-1} \frac{x}{t^{\mu}}\right),  \label{eq:selfsimsol}
\end{equation}
where we used the following relations: 
\[
\int_{\R_+}g_\mu(\eta) d\eta = \left(\int_{\R_+}f_\mu(\xi) d\xi\right)^2
= t^{-2\mu+1}\int_{\R_+}u(t,x)dx.
\]
 The following Proposition guaranties the existence of self-similar
solutions of the form \eqref{eq:selfsimsol} for any $\mu \in [\frac{1}{3},1)$%
.

\begin{prop}
\label{lem:solf} For $\mu \in \left[ \frac{1}{3},1\right) $
there exists a solution to \eqref{ss1} such that $f(0)=0$ and $f(\xi
)\rightarrow 0$ as $\xi \rightarrow \infty $. The solution is positive and
such that if $\mu \in \left( \frac{1}{3},1 \right) $,
\begin{equation*}
f_\mu(\xi )=O(\xi ^{1-\frac{1}{\mu }})\textrm{ as }\xi \rightarrow \infty,
\end{equation*}%
%
and for $\mu =\frac{1}{3}$,
\begin{equation*}
f_{\frac{1}{3}}(\xi )=\xi e^{-\frac{1}{6}\xi ^{2}}.
\end{equation*}
\end{prop}

\begin{proof}
Let $\mu\in \left[ \frac{1}{3},1\right)$. We look for an
analytic solution to \eqref{ss1}, of the form
\begin{equation*}
f_\mu(\xi )=\sum_{k=0}^{+\infty}b_{k}\xi ^{k},
\end{equation*}%
where $b_k\in\R$ for all $k\in\N$. The boundary condition in $x=0$ implies
that $f_\mu(0)=b_0=0$. From \eqref{ss1}, we obtain
\begin{equation*}
\sum_{k=0}^{+\infty}\left[ (\mu-1) b_k - \mu k b_k - (k+2)(k+1) b_{k+2}%
\right] \xi ^{k} = 0.
\end{equation*}
Thus, for any $k\in \N$, it holds
\begin{equation*}
b_{k+2} = \frac{\mu - 1 - \mu k}{(k+1)(k+2)}b_k.
\end{equation*}
In particular, the condition $b_0=0$ implies that $b_{2n}=0$ for all $n\in\N$%
. Thus, denoting $w_n:=b_{2n+1}$ we can rewrite $f$ as
\begin{equation*}
f_\mu(\xi )=\sum_{n=0}^{+\infty}w_n\xi ^{2n+1},
\end{equation*}%
where $(w_n)_{n\in\N}$ satisfy the following relation (since $\mu>0$): 
\begin{equation*}
a_{n+1}=-\frac{\mu }{2}\frac{n+\frac{1}{2\mu }}{\left( n+\frac{3}{2}\right)
(n+1)}w_n
\end{equation*}%
and hence%
\begin{equation*}
w_n=(-1)^{n}\left( \frac{\mu }{2}\right) ^{n}\frac{\Gamma \left( \frac{3}{2%
}\right) \Gamma (1)}{\Gamma \left( \frac{1}{2\mu }\right) }\frac{\Gamma
\left( n+\frac{1}{2\mu }\right) }{\Gamma \left( n+\frac{3}{2}\right) \Gamma
(n+1)}a_{0}.
\end{equation*}


The solution can be written in terms of classical hypergeometric confluent
functions $_{1}F_{1}$:%
\begin{eqnarray*}
f_\mu(\xi ) =\xi\; _{1}F_{1}\left( \frac{1}{2\mu },\frac{3}{2};-\frac{\mu }{2%
}\xi ^{2}\right) =\xi e^{-\frac{\mu}{2}\xi ^{2}}\phantom{.}_{1}F_{1}\left(
\frac{3}{2}-\frac{1}{2\mu },\frac{3}{2}; \frac{\mu }{2}\xi ^{2}%
\right) .
\end{eqnarray*}%
In the particular case $\mu =\frac{1}{3}$ one has%
\begin{equation*}
f_{\frac{1}{3}}(\xi )=\xi e^{-\frac{1}{6}\xi ^{2}}
\end{equation*}%
while, for $\mu \in (\frac{1}{3},1)$  (cf. \cite{MR1225604} formula 13.7.1)%
\begin{equation*}
\phantom{.}_{1}F_{1}\left( \frac{1}{2\mu },\frac{3}{2};-\frac{\mu }{2}\xi
^{2}\right) \sim \frac{\Gamma \left( \frac{3}{2}\right) }{\Gamma \left( %
\frac{3}{2}-\frac{1}{2\mu}\right) }\left( \frac{\mu }{2}\xi
^{2}\right) ^{-\frac{1}{2\mu }}\text{ as }\xi \rightarrow \infty ,
\end{equation*}%
that is
\begin{equation*}
f_\mu(\xi )=O(\xi ^{1-\frac{1}{\mu }}).
\end{equation*}

The positivity of $f_\mu(\xi )$ follows from the positivity of the integrand
in the following representation formula for the confluent hypergeometric
function of the first kind:%
\begin{equation*}
\phantom{.}_1F_{1}\left( \alpha,\beta;z\right) =\frac{\Gamma (
\beta)}{\Gamma (\beta-\alpha)\Gamma ({\alpha})}%
\int_{0}^{1}e^{zt}t^{\alpha-1}(1-t)^{ \beta-\alpha-1}\,\mathrm{%
d} t
\end{equation*}%
with $\alpha=\frac{1}{2\mu}$, $\beta=\frac{3}{2}$. This concludes the proof
of the Lemma.
\end{proof}

\begin{rem}
For $\mu<\frac{1}{3}$, $f_\mu(\xi )$ is not positive and has a zero, $%
\xi^0_\mu$, that comes from infinity as $\mu$ decreases and approaches $\xi^0
_{0}(0)=\pi $ (since $f_0(\xi )=\sin (\xi) $ for $\mu=0$).
\end{rem}

 Proposition \ref{lem:solf} thus provides us with a family of
solutions $f_\mu$ to equation \eqref{ss1}, for $\mu\in \left[\frac{1}{3}%
,1\right)$. Recall that by definition, the solution to \eqref{eq:cont1}-%
\eqref{eq:cont3} must have finite mass. The relation between $u$ and $f_\mu$
given by \eqref{eq:selfsimsol} implies that for all $t\in\R_+$,
\begin{equation*}
\int_{\R_+} u(t,x) \mathrm{d}x = t^{2\mu-1} \left(\int_{\R_+} f_\mu\left(\xi\right)
d\xi\right)^2.
\end{equation*}
From Lemma \ref{lem:solf}, $f_\mu$ is not integrable for any $\mu\in \left(%
\frac{1}{3},1\right)$, which means that the only admissible solution to %
\eqref{ss1} giving a self-similar solution to \eqref{eq:cont1}-%
\eqref{eq:cont3} with finite mass is given by $\mu=\frac{1}{3}$.

We then postulate that the self-similar solution $f_{\frac{1}{3}}$ for $\mu =%
\frac{1}{3}$ is an attractor, in the sense that solutions tend to the
self-similar solution in a suitable norm as $t\rightarrow \infty $, for all
initial data that decay sufficiently fast.

The remainder of this article aims to prove that this is indeed the case. We
begin by defing a quantity that plays an important role in the definition of
$u$ and in the analysis of its asymptotic behavior. Given a solution $u$ to %
\eqref{eq:cont1}-\eqref{eq:cont3}, and its first moment $M:t\rightarrow M(t)$%
, we define
\begin{equation}  \label{eq:a}
a(t):= \int_0^t M(s)\,\mathrm{d} s.
\end{equation}

\color{black}

The question now is the identification of $a(t)$ given by \eqref{eq:a}. Note
that from \eqref{eq:usol},
\begin{equation*}
u({t,x})=\frac{1}{{2}\sqrt{\pi a(t)}}\int_{0}^{+\infty }\left( e^{-\frac{%
(x-s)^{2}}{4a(t)}}-e^{-\frac{(x+s)^{2}}{4a(t)}}\right) u^{\mathrm{in}} (s)\,%
\mathrm{d} s
\end{equation*}%
so that%
\begin{equation*}
\partial_x u({t,0})=\frac{1}{{2}\sqrt{\pi }}a^{-\frac{3}{2}}(t)\int_{0}^{+\infty
} se^{-\frac{s^{2}}{{4} a(t)}}u^{\mathrm{in}}(s)\,\mathrm{d} s.
\end{equation*}%
Integrating Equation \eqref{p:reg} in $\mathbb{R}_{+}$, as seen in the proof of Proposition
\ref{p:mass}, it holds
\begin{equation*}
\frac{\mathrm{d} M(t)}{\mathrm{d} t}=-M(t)\partial_x u({t,0}),
\end{equation*}%
which allows us to conclude that $a(t)$ satisfies the following integro-differential
equation:%
\begin{equation}
a^{\prime \prime }(t)=-a^{\prime }(t)\frac{1}{{2}\sqrt{\pi }}a^{-\frac{3}{2}%
}(t)\int_{0}^{+\infty } se^{-\frac{s^{2}}{{4}a(t)}}u^{\mathrm{in}}(s)\,\mathrm{%
d} s.  \label{intdif}
\end{equation}%
We define now
\begin{equation*}
G(a):=\frac{1}{{2}\sqrt{\pi }}a^{-\frac{3}{2}}\int_0^{+\infty} se^{-\frac{s^{2}}{{4}a}}u^{\mathrm{in}}(s)\,\mathrm{d} s.
\end{equation*}
The quantity $G(a)$ is bounded provided that $u^{\mathrm{in}}$ is linear at
the origin and has its first moment $M_1(0)$ bounded.

\begin{lemma}
\label{lem:a} Let $u^{\mathrm{in}}\in L^1(\R_+)\cap L^\infty(\R_+)$ a positive and 
admissible initial condition.
The
integro-differential equation \eqref{intdif}, with initial condition $a(0)=0$ and
 $a^{\prime }(0)=M(0)$, has a solution $a:\,\R_+\to \R$ such
that%
\begin{equation*}
a(t)\sim\left( {\frac{3}{2\sqrt{\pi}}}M_{1}\right) ^{\frac{2}{3}}t^{\frac{2}{%
3}}\text{ as }t\rightarrow \infty
\end{equation*}%
and
\begin{equation*}
a^{\prime }(t)\sim\frac{2}{3}\left( {\frac{3}{2\sqrt{\pi}}}M_{1}\right) ^{%
\frac{2}{3}}t^{-\frac{1}{3}}\text{ as }t\rightarrow \infty.
\end{equation*}
\end{lemma}

\begin{proof}
Since
\begin{equation*}
a^{\prime \prime }(t)=-\frac{d}{\,\mathrm{d} t}\int_{0}^{a(t)}\left( \frac{1%
}{{2}\sqrt{\pi }}(a^{\ast})^{-{3}/{2}}\int_{0}^{+\infty}se^{-\frac{s^{2}}{%
4a^{\ast }}}u^{\mathrm{in}}(s)\,\mathrm{d} s\right) \,\mathrm{d} a^{\ast }
\end{equation*}%
Integrating once in time and using that $a^{\prime }(0)=M({0})$, it holds%
\begin{equation*}
a^{\prime }(t)+\int_{0}^{a(t)}\left( \frac{1}{{2}\sqrt{\pi }}(a^{\ast})^{-{3}/{%
2}}\int_{0}^{+\infty}se^{-\frac{s^{2}}{4a^{\ast }}}u^{\mathrm{in}}(s)\,%
\mathrm{d} s\right)\, \mathrm{d} a^{\ast }=M(0),
\end{equation*}%
which we rewrite as 
\begin{equation}  \label{eq:aF}
a^{\prime }(t)+F(a(t))=M({0}), 
\end{equation}%
denoting $F(a):=\int_{0}^{a}G(a^\ast)$.
But now%
\begin{equation*}
\begin{split}
F(a)=&\int_{0}^{a}G(a^\ast)\,\mathrm{d} a^\ast =\int_{0}^{a}\left( \frac{1}{{2}\sqrt{\pi }%
}(a^{\ast})^{-{3}/{2}}\int_{0}^{+\infty}se^{-\frac{s^{2}}{4a^{\ast }}}u^{%
\mathrm{in}}(s)\,\mathrm{d} s\right) \, \mathrm{d} a^{\ast } \\
=& \frac{1}{{2}\sqrt{\pi }}\int_{0}^{+\infty}\left( \int_{0}^{a}(a^{\ast})^{-{3%
}/{2}}e^{-\frac{s^{2}}{4a^{\ast }}}\, \mathrm{d} a^{\ast }\right) su^{\mathrm{in}}(s)\,%
\mathrm{d} s =\frac{1}{{2}\sqrt{\pi }}\int_{0}^{+\infty}\left(
\int_{0}^{as^{-2}}(a^{\ast})^{-{3}/{2}}e^{-\frac{1}{4a^{\ast }}}\, \mathrm{d} a^{\ast
}\right) u^{\mathrm{in}}(s)\,\mathrm{d} s
\end{split}%
\end{equation*}%
and, since $\int_{0}^{+\infty}a^{-\frac{3}{2}}e^{-\frac{1}{4a}}da={2}\sqrt{\pi}$, we obtain
\begin{equation*}
\lim_{a^\star\rightarrow+\infty} F(a^\star )=M({0}).
\end{equation*}%
We conclude that as $t$ tends to infinity, if $a(t)\rightarrow +\infty $,
then $a^{\prime }(t)\rightarrow 0$.

{ Notice that from its definition, $a$ is an increasing
function, hence it has a limit when $t$ goes to infinity. Let $%
a_\infty:=\lim_{t\rightarrow+\infty} a(t)$,  and suppose that $%
a_\infty<+\infty$. Then $\lim_{t\rightarrow+\infty} a^\prime(t) = 0$, from
which we get $F(a_\infty)=M(0)$. However, $F(a)$ is the primitive of a
strictly positive function and hence is strictly growing as a function of $a$%
, which contradicts $\lim_{a^\star\rightarrow+\infty} F(a^\star )=M({0})$. }

We then conclude that $a(t)\rightarrow +\infty $ as $t\rightarrow +\infty $.
Then, as $a(t)\rightarrow +\infty $, $G(a(t))\sim \frac{1}{{2}\sqrt{\pi }}%
a(t)^{-\frac{3}{2}}M_{1} $ and
\begin{equation*}
a^{\prime \prime }(t)\sim -a^{\prime }(t)a^{-\frac{3}{2}}(t)\frac{1}{{2}\sqrt{\pi }}M_{1},
\end{equation*}%
so that $a(t)\sim ct^{\frac{2}{3}}\text{ as }t\rightarrow \infty $, with $-%
\frac{2}{9}c=-\frac{{1}}{3}c^{-\frac{1}{2}}\sqrt{\frac{1}{\pi }}M_{1}$, that
is
\begin{equation*}
c=\left( {\frac{3}{2\sqrt{\pi}}}M_{1}\right) ^{\frac{2}{3}}.
\end{equation*}%
This concludes the proof of the Lemma.
\end{proof}

On the other hand, if $u^{\mathrm{in}}$ \textit{does not have its first
moment bounded} but%
\begin{equation*}
u^{\mathrm{in}}(x)=O(x^{-\delta })\text{, as }x\rightarrow +\infty ,\quad 1%
\mathit{<\delta <2}
\end{equation*}%
\ then%
\begin{equation*}
G(a)=
\frac{1}{2\sqrt{\pi }}
a^{-\frac{3}{2}}\int_{{0}}^{{+\infty}} 
\!\!\!\!
se^{-\frac{s^{2}}{{4}a}}u^{%
\mathrm{in}}(s)\,\mathrm{d}s\sim 
{\frac{1}{\sqrt{\pi }}}
a^{-\frac{1}{2}}\int_{{0}}^{{+\infty}} 
\!\!\!\!
se^{-\frac{s^{2}}{2}}u^{\mathrm{in}}({\sqrt{2}}a^{\frac{1}{2}}s)\,\mathrm{d}s\sim Ca^{-
\frac{1}{2}-\frac{\delta }{2}}\int_{0}^{+\infty }
\!\!\!\!
e^{-\frac{s^{2}}{2}%
}s^{1-\delta }\,\mathrm{d}s
\end{equation*}%
and hence,%
\begin{equation*}
a^{\prime \prime }(t)\sim -Ca^{\prime }a^{-\frac{1}{2}-\frac{\delta }{2}}
\end{equation*}%
implying%
\begin{equation*}
a(t)=O(t^{\frac{2}{\delta+1 }})
\end{equation*}%
so that
\begin{equation*}
\mu =\delta +1.
\end{equation*}
 Lemma \ref{lem:a} thus gives us the asymptotic behavior of $a$
as $t$ goes to infinity. It has two consequences. In Proposition \ref%
{p:unorm}, we prove that the $L^{\infty }$-norm of the solution to %
\eqref{eq:cont1}-\eqref{eq:cont3} decays like $t^{-\frac{2}{3}}$. We can
then compare the asymptotic behavior of $u$ to that of the candidate 
self-similar profile.
Proposition \ref{p:profilenorm} shows that this profile indeed also decays like $t^{-\frac{2}{3}}$.

\begin{prop}
\label{p:unorm} Let $u^{\mathrm{in}}\in L^1(\R_+)\cap L^\infty(\R_+)$ be a
positive and admissible initial condition.
Then for all $x\in\R_+$,
\begin{equation}
u({t,x}) \leq \frac{M_1}{\sqrt{2 e\,\pi }\, a(t)} \sim \frac{C M_1^{\frac{1}{3}%
}}{t^{\frac{2}{3}}} \quad \text{ as } t\rightarrow+\infty,
\end{equation}
where the value of $C$ can be computed explicitely and does not depend on
the initial data.
\end{prop}

\begin{proof}
From Equation \eqref{eq:usol},
\begin{equation*}
u({t,x})=\frac{1}{{2}\sqrt{\pi a(t)}}\int_{0}^{+\infty } f_{a(t)}(s) u^{\mathrm{%
in}} (s)\,\mathrm{d} s,
\end{equation*}
where
\begin{equation*}
f_{a}(s) := e^{-\frac{(x-s)^{2} }{4a}}-e^{-\frac{(x+s)^{2}}{4a}}.
\end{equation*}
Notice that
\begin{equation*}
f_{a}(s) = g_a(s)-g_a(-s) = \int_{-s}^s g_a^\prime(\xi)\, \mathrm{d}\xi,
\end{equation*}
where $g_a(\xi) = e^{-\frac{(x-\xi)^{2} }{4a}}.$ One easily sees that
\begin{equation*}
g_a^\prime(\xi) = \frac{1}{\sqrt{a}}\left( \frac{x-\xi }{\sqrt{4a}}e^{-\frac{%
(x-\xi)^{2} }{4a}} \right) \leq \frac{1}{\sqrt{2 e\, a}},
\end{equation*}
in which we used the property: $|ze^{-z^2}|\leq (2e)^{-\frac{1}{2}}$ for all
$z\in\R_+$. Hence, $f_a(s)\leq \frac{2s}{\sqrt{2 e\, a}}$, which implies
that
\begin{equation*}
u({t,x}) \leq \frac{1}{\sqrt{2 e\,\pi }\, a(t)} M_1.
\end{equation*}
Lemma \ref{lem:a} allows us to conclude.
\end{proof}

\begin{prop}
\label{p:profilenorm} Let $u^{\mathrm{in}}\in L^1(\R_+)\cap L^\infty(\R_+)$
be a positive {and admissible initial condition}.
Let $u$ be the solution to %
\eqref{eq:cont1}-\eqref{eq:cont3}, and let $$a:t\mapsto \int_0^{+\infty}
u({t,x})\,\mathrm{d}x.$$ Then for all $x\in\R_+$,
\begin{equation}
\frac{M_{1}x}{{2}\sqrt{\pi }a^{\frac{3}{2}}(t)}e^{-\frac{x^{2}}{4a(t)}} \sim
\frac{C}{t^{\frac{2}{3}}} \quad \text{ as } t\rightarrow+\infty,
\end{equation}
for some positive constant $C$.
\end{prop}

\begin{proof}
For all $x\in\R_+$,
\begin{equation*}
\frac{M_{1}x}{2\sqrt{\pi }a^{\frac{3}{2}}(t)}e^{-\frac{ x^{2}}{4a(t)}} =
\frac{M_{1}}{\sqrt{\pi }a(t)}\frac{x}{2\sqrt{a(t) }}e^{-\left(\frac{x}{\sqrt{%
4a(t)}}\right)^2} \leq \frac{M_{1}}{\sqrt{2\, e\,\pi }a(t)} \sim \frac{C}{t^{%
\frac{2}{3}}},
\end{equation*}
using the asymptotic behavior of $a$ shown in Lemma \ref{lem:a}.
\end{proof}

Propositions \ref{p:unorm} and \ref{p:profilenorm} show that the $L^\infty$%
-norms of the solution $u$ to \eqref{eq:cont1}-\eqref{eq:cont3} and of the
candidate self-similar solution 
decay with the
same order. In the final theorem of this article we show that 
$u$ asymptotically approaches the
self-similar solution as soon as the
second moment of the initial data is bounded.

\begin{theorem}
\label{th:convergence} If the initial data $u^{\mathrm{in}}$ has a bounded second
moment $M_{2}(0)$, then there exists $C>0$ such that
\begin{equation*}
\left\vert u({t,x})-M_{1}\frac{x}{{2}\sqrt{\pi }a^{\frac{3}{2}}(t)}e^{-\frac{%
x^{2}}{4a(t)}}\right\vert \leq \frac{CM_{2}(0)}{t}
\end{equation*}%
for $t>1$.
\end{theorem}

\begin{proof}
Since
\begin{equation*}
u({t,x})=\frac{1}{{2}\sqrt{\pi a(t)}}\int_{0}^{+\infty}\left( e^{-\frac{(x-y)^{2}%
}{4a(t)}}-e^{-\frac{(x+y)^{2}}{4a(t)}}\right) u^{\mathrm{in}}(y)\,\mathrm{d} y,
\end{equation*}%
denoting%
\begin{equation*}
v (x)=\frac{1}{{2}\sqrt{\pi a}}\int_{0}^{+\infty}\left( \frac{e^{-\frac{
(x-y)^{2}}{4a}}-e^{-\frac{(x+y)^{2}}{4a}}}{y}\right) yv^{\mathrm{in}}(y)\,\mathrm{d} y,
\end{equation*}
we have
\begin{eqnarray*}
&&v (x)-M_{1}\frac{x}{{2}\sqrt{\pi }a^{\frac{3}{2}}}e^{-\frac{x^{2}}{4a}} \\
&=&\frac{1}{{2}\sqrt{\pi a}}\int_{0}^{+\infty}\left( \frac{e^{-\frac{(x-y)^{2}%
}{4a}}-e^{-\frac{(x+y)^{2}}{4a}}}{y}-\frac{x}{a}e^{-\frac{x^{2}}{4a}}\right)
yv^{\mathrm{in}}(y)\,\mathrm{d}y.
\end{eqnarray*}%
We write now%
\begin{equation*}
\frac{e^{-\frac{(x-y)^{2}}{4a}}-e^{-\frac{(x+y)^{2}}{4a}}}{y}-\frac{x}{a}e^{-%
\frac{x^{2}}{4a}}\equiv \frac{1}{a^{\frac{1}{2}}}{\Phi}\left( \frac{x}{a^{\frac{1%
}{2}}},\frac{y}{a^{\frac{1}{2}}}\right)
\end{equation*}%
with%
\begin{equation*}
{\Phi}\left( X,Y\right) =\frac{e^{-\frac{(X-Y)^{2}}{4}}-e^{-\frac{(X+Y)^{2}}{4}}}{%
Y}-Xe^{-\frac{X^{2}}{4}}.
\end{equation*}

It is simple to show that there exists a constant $C$ such that%
\begin{equation*}
\left\vert {\Phi}\left( X,Y\right) \right\vert \leq CY
\end{equation*}%
so that%
\begin{equation*}
\left\vert v(x)-M_{1}\frac{x}{{2}\sqrt{\pi }a^{\frac{3}{2}}}e^{-\frac{x^{2}}{4a%
}}\right\vert \leq \frac{C}{a^{\frac{3}{2}}}\int_{0}^{+\infty}y^{2}u^{%
\mathrm{in}}(y)dy.
\end{equation*}
\end{proof}

{\ \color{blue} }

\begin{rem}
Note that the previous result can be rewritten as
\begin{equation*}
t^{2/3}\left\vert u({t,x})-M_{1}\frac{x}{{2}\sqrt{\pi }a^{\frac{3}{2}}(t)}e^{-%
\frac{x^{2}}{4a(t)}}\right\vert \leq \frac{CM_{2}(0)}{t^{1/3}},
\end{equation*}
which means that the convergence of the solution to the self-similar profile
takes place at a faster rate than the decay of their $L^\infty$-norms, which
is to be expected.
\end{rem}


\section{Numerical tests}
\label{s:num}

In this section we perform some numerical experiments in order to verify the
theoretical results obtained above.

At the numerical level, we worked with the finite space interval $[0,400]$,
which is sufficiently wide to minimize the boundary effects on the numerical
solution, especially for initial data having a fast decay when $x\rightarrow
+\infty $.

We have introduced a fixed space step $\Delta x>0$ and a time step $\Delta
t>0$. Then, we have divided the interval $[0,400]$ in $N$ sub-intervals of
measure $\Delta x=400/N$.

We have then used an explicit finite differences scheme where the diffusion
coefficient (i.e. the mass) at each time step is taken as the mass in the
previous time step. The method is stable under the standard stability
condition $\Delta t \leq M(0) (\Delta x)^2/2$.

\begin{figure}[t]
\includegraphics[width=.7\textwidth]{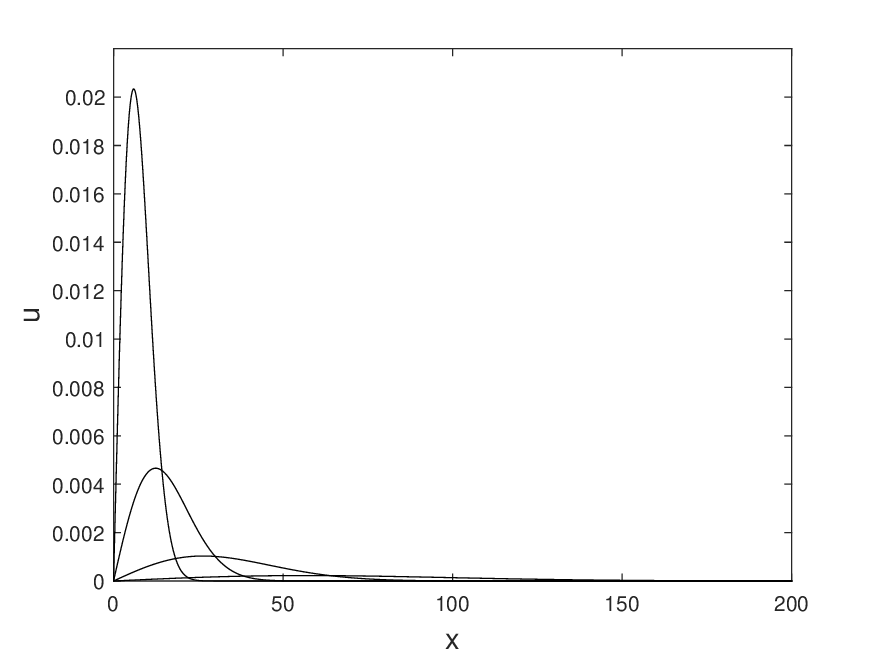}
\caption{Numerical profiles of the solution of \eqref{eq:cont1}-\eqref{eq:cont3} at times $t=50,\
500,\ 5000,\ 50000$, with initial
condition given in \eqref{e:ictest}.}
\label{fig1}
\end{figure}

\begin{figure}[t]
\includegraphics[width=.7\textwidth]{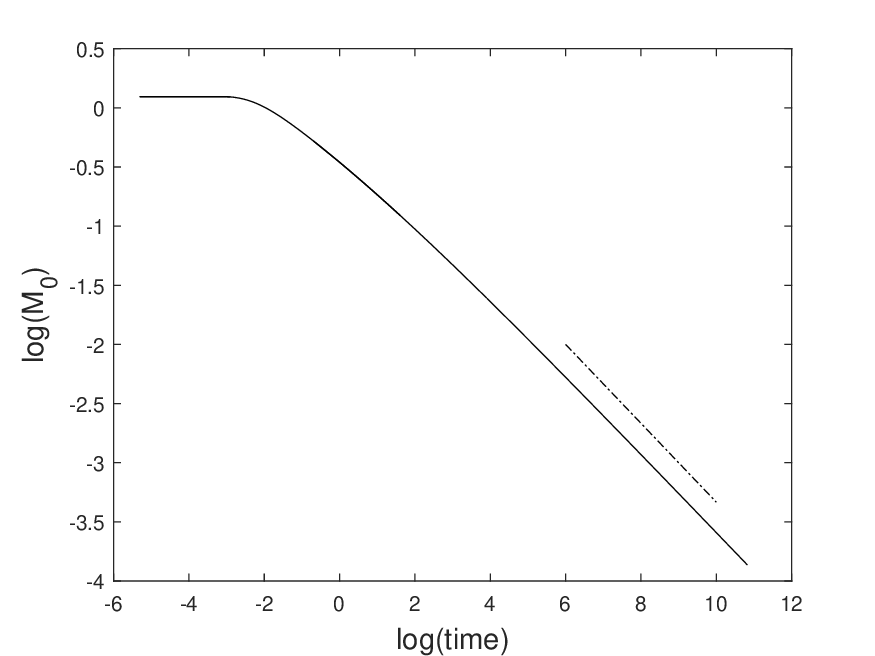}
\caption{Logarithm of the mass vs logarithm of time and comparison with a $-1/3$ slope (dotted-dashed line).}
\label{fig2}
\end{figure}

We have taken, as initial data, 
\begin{equation}
\label{e:ictest}
u_{0}(x)=\chi _{\lbrack 1,2]}=
\left\{
\begin{array}{ll}
1 & x\in \left[ 1,2\right] \\[6pt]
0 & \mathrm{otherwise}.
\end{array}
\right.
\end{equation}
In Figure \ref%
{fig1}, we plot the numerical approximation of the solution $u$ for $t=50,\
500,\ 5000,\ 50000$ and, in Figure \ref{fig2}, we show the time evolution of the
quantity $\log (M_{0}(t))$ (i.e. the logarithm of the mass of $u$). As we can see,
$\log (M_{0}(t))$ tends to follow a straight line with slope $-\frac{1}{3}$,
which indicates an asympotics of type $M_{0}(t)=O(t^{-1/3})$ as $%
t\rightarrow \infty $. Finally, in Figure \ref{fig3} we rescale the profiles in Figure \ref{fig1}
by multiplying them by $a(t)$ and representing them as a function of $x/a^{%
\frac{1}{3}}(t)$. As we can see, they approach the self-similar profile $%
f(\eta )=\frac{M_{1}}{\sqrt{4\pi }}\eta e^{-\frac{\eta ^{1}}{4}}$ (dashed
line).

\begin{figure}[t]
\includegraphics[width=.7\textwidth]{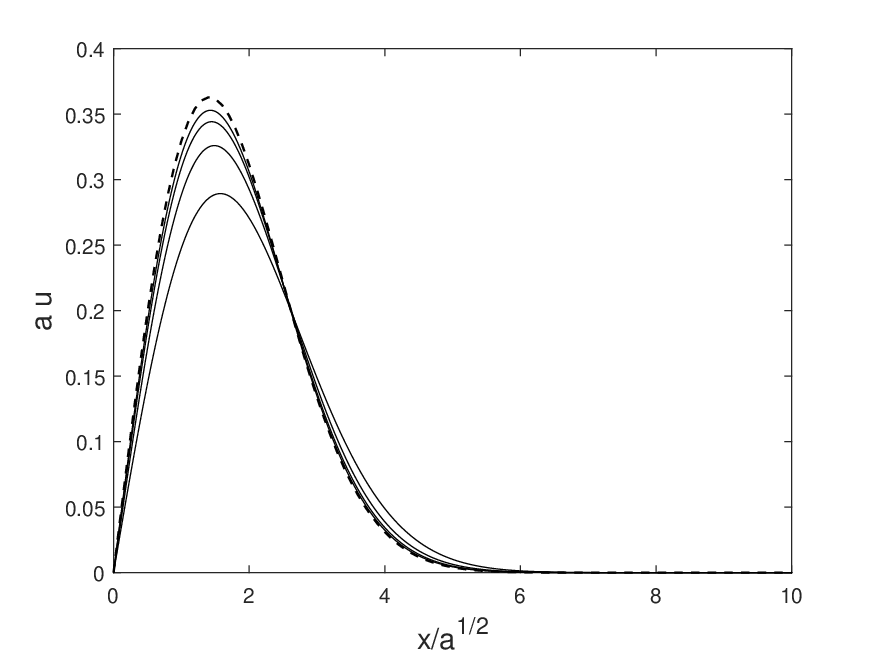}
\caption{Rescaled numerical profiles and comparison with the self-similar
profile (dashed line).}
\label{fig3}
\end{figure}


\section*{Appendix}


We consider here the case of the finite domain $\Omega=(0,\pi)$, with
homogeneous Dirichlet boundary conditions. This setting describes the
diffusive limit of the kinetic rock-paper-scissors game by supposing that
only individuals with wealth $x\in\Omega$ play the game, and can be deduced
from the kinetic model described in \cite{MR4128577} by adapting the same
arguments.

The problem studied in this Appendix is hence the following. We consider the
equation
\begin{equation}  \label{e:bdd1}
w_t=\left [\int_{0}^\pi w(t,\xi)\,\mathrm{d} \xi\right] w_{xx}, \qquad
(t,x)\in\R_+\times \Omega
\end{equation}
with initial data
\begin{equation}  \label{e:bdd2}
w(0,x)= w^{\text{in}}\in L^{2}( 0,\pi) \qquad
x\in\Omega
\end{equation}
and boundary conditions
\begin{equation}  \label{e:bdd3}
w(t,0)=w(t,\pi)=0, \qquad t\in\R_+,
\end{equation}
 where $w^{\text{in}}\geq 0$ for a.e. $x\in (0,\pi)$.
Note that, by parabolic theory 
\cite{MR0241822}, $w(t,x)\geq 0$. 

An explicit solution of \eqref{e:bdd1}-\eqref{e:bdd3}, when $w^{\text{in}%
}=M\sin(x)$, is the following:
\begin{equation*}
w^*(\textcolor{black}{t,x})=\frac{M}{2} \frac{\sin (x)}{1+Mt},\quad (t,x)\in\R_+\times \Omega,
\end{equation*}
where $M>0$ is a given constant. The function $w^*$ also turns out to be a
self-similar solution with the similarity exponent $\mu=0$. Its initial mass
is
\begin{equation*}
\int_{0}^\pi w^*(t,\xi)\,\mathrm{d} \xi = M.
\end{equation*}

We will show that, indeed, if the initial data $w^{\text{in}}$ is in $L^{2}(\Omega )$, the solution will tend to the explicit solution,
i.e. 
$$
w(t,x)\sim \frac{M\sin x}{2(1+Mt)} \text{ as } t\rightarrow + \infty,
$$
and the rate
of convergence is $\mathcal{O}(t^{-2})$. Clearly, $w\in C(\R%
_{+};L^{2}(\Omega ))$ and we can write $w(t,x)$ in terms of Fourier series
which, because of the boundary conditions, takes the form
\begin{equation}
w(t,x)=\sum_{n=1}^{+\infty }{w_n(t)}\sin (nx).  \label{wtx}
\end{equation}%
By simple inspection in \eqref{e:bdd1}-\eqref{e:bdd3}, we deduce that $w$
solves \eqref{e:bdd1}, with initial condition
\[
w^{\text{in}}=\sum_{n=1}^{+\infty }{w_n(0)}\sin (nx),
\]%
and boundary conditions \eqref{e:bdd3}. Let%
\[
M(t)\equiv \int_{0}^{\pi }w(t,x)dx.
\]%
Then%
\[
M(t)=\sum_{n\ \text{odd}}\frac{2w_n(t)}{n}
\]%
so that%
\[
\frac{dw_n}{dt}=-n^{2}M(t)w_n.
\]%
Hence%
\[
w_n(t)=w_n(0)e^{-n^{2}\int_{0}^{t}M(t^{\prime })dt^{\prime }},
\]%
and we can compute
\[
M(t)=\sum_{n\ \text{odd}}\frac{2w_n(0)}{n}e^{-n^{2}\int_{0}^{t}M(t^{\prime
})dt^{\prime }}.
\]%
Denoting, as before, 
$$a
(t)=\int_{0}^{t}M(t^{\prime })dt^{\prime },
$$ 
we have then the
ordinary differential equation
\[
a^{\prime }(t)=\sum_{n\ \text{odd}}\frac{2w_n(0)}{n}e^{-n^{2}a(t)}
\]%
so that, since $w(t,x)\geq 0 $ and hence $M(t)>0$, we can integrate explicitly to obtain%
\[
G(a)\equiv \int_{0}^{a}\frac{e^{a^{\prime }}}{{ \sum_{n\ \text{odd}}}\frac{2w_n(0)}{n%
}e^{(1-n^{2})a^{\prime }}}da^{\prime }=t.
\]%
Note that%
\begin{eqnarray*}
G(a) &=&\frac{1}{2w_1(0)}(e^{a}-1)-\frac{1}{2w_1(0)}\int_{0}^{a}\frac{%
\sum_{n=3,5,...}\frac{w_n(0)}{nw_1(0)}e^{(2-n^{2})a^{\prime }}}{%
1+\sum_{n=3,5,...}\frac{w_n(0)}{nw_1(0)}e^{(1-n^{2})a^{\prime }}}%
da^{\prime } \\
&=&\frac{1}{2w_1(0)}e^{a}-K+O(e^{-7a}),\ \text{as }a\rightarrow +\infty
\end{eqnarray*}%
with%
\[
K=\frac{1}{2w_1(0)}+\frac{1}{2w_1(0)}\int_{0}^{\infty }\frac{%
\sum_{n=3,5,...}\frac{w_n(0)}{nw_1(0)}e^{(2-n^{2})a^{\prime }}}{%
1+\sum_{n=3,5,...}\frac{w_n(0)}{nw_1(0)}e^{(1-n^{2})a^{\prime }}}%
da^{\prime }
\]%
We have then%
\[
a\sim \log (2w_1(0)(t+K)+O(t^{-7})),\ \text{as }t\rightarrow + \infty ,
\]%
and hence%
\[
M(t)=a^{\prime }\sim \frac{1}{t+K},\ \text{as }t\rightarrow +\infty .
\]%
Therefore%
\begin{equation}
w_n(t)\sim \frac{w_n(0)}{(2w_1(0)(t+K))^{n^{2}}}\text{ as }%
t\rightarrow +\infty .  \label{at}
\end{equation}

We can prove then the following Lemma:

\begin{lemma}
Let $w^{\text{in}}$ be the initial condition of the initial-boundary value
problem \eqref{e:bdd1}-\eqref{e:bdd3} and suppose that $w^{\text{in}}\in
L^{1}(\Omega )\cap L^{\infty }(\Omega )$. Then there exists a constant $C$, depending on  $w^{\text{in}}$, and a
time $T>0$ such that, for any $t>T$,
\[
\left\vert w(t,x)-
\frac{M}{2} \frac{\sin (x)}{1+Mt}
\right\vert \leq \frac{C}{t^{2}}.
\]
\end{lemma}

\begin{proof}
We note%
\[
w_n(0)=\frac{2}{\pi }\int_{0}^{\pi }w^{\text{in}}(x)\sin (nx)\,\mathrm{d}x,
\]
so that
\[
\left\vert w_n(0)\right\vert \leq \frac{2}{\pi }\int_{0}^{\pi }\left\vert
w^{\text{in}}(x)\right\vert \,\mathrm{d}x=\frac{2}{\pi }\left\Vert w^{\text{%
in}}\right\Vert _{L^{1}(\Omega )}
\]%
and use (\ref{wtx}), (\ref{at}).
\end{proof}

\bigskip \noindent{\bf{Acknowledgments:}} This article has been written when the
third author was visiting the Instituto de Ciencias Mat\'ematicas (ICMAT) in
Madrid. FS deeply thanks ICMAT for its hospitality.

The first author is supported by the project PID2020-113596GB-I00.
The second author benefited from the Emergence grant EMRG-33/2023 of Sorbonne University. 
The third author acknowledges the support of INdAM, GNFM group, and of the COST Action CA18232 MAT-DYN-NET, supported by COST (European Cooperation in Science and Technology).

The authors thank the anonymous referee for his/her remarks and suggestions which helped us in improving our paper.

\bibliographystyle{plain}
\bibliography{biblio.bib}

\end{document}